\documentclass{article}
\usepackage[utf8]{inputenc}
\usepackage{fullpage}
\usepackage{amsmath,amsfonts,amssymb,amsthm,enumerate,mdframed,hyperref,cleveref,algorithmic,algorithm,multirow,import,tikz,booktabs,xcolor,soul}
\usetikzlibrary{arrows,positioning, shapes} 
\usepackage{caption,subcaption}
\usepackage{enumitem}
\usepackage{authblk}

\newtheorem{theorem}{Theorem}

\newtheorem{lemma}[theorem]{Lemma}

\newtheorem{definition}[theorem]{Definition}
\newtheorem{corollary}[theorem]{Corollary}

\mathcode`*=\numexpr\mathcode`*-"2000\relax

\definecolor{darkgreen}{rgb}{0,0.6,0}

\DeclareMathOperator{\PG}{PG}
\DeclareMathOperator{\supp}{supp}

\title{Trifferent codes with small lengths}
\author{Sascha Kurz}
\affil{Mathematisches Institut, Universität Bayreuth, D-95440 Bayreuth, Germany, sascha.kurz@uni-bayreuth.de}
\date{}

\begin{document}

\maketitle

\begin{abstract}
  A code $C \subseteq \{0, 1, 2\}^n$ of length $n$ is called trifferent if for any three distinct elements of $C$ there exists a coordinate in which they
  all differ. By $T(n)$ we denote the maximum cardinality of trifferent codes with length. $T(5)=10$ and $T(6)=13$ were recently determined  \cite{della2022maximum}. 
  Here we determine $T(7)=16$, $T(8)=20$, and $T(9)=27$. For the latter case $n=9$ there also exist linear codes attaining the maximum possible cardinality $27$.
\end{abstract}

\section{Introduction}
For integers $k\ge 3$ and $n\ge 1$ any set $C\subseteq \{0,1,\dots,k-1\}^n$ is called a $k$-ary \emph{code} of length $n$. If $C$ has the
property that for any $k$ distinct elements there exists a coordinate in which they all differ then $C$ is called \emph{perfect $k$-hash code}.
The problem of determining the largest possible size of a perfect $k$-hash code is a natural combinatorial problem with connections to 
topics in cryptography, information theory, and computer science, see e.g.\ \cite{della2022improved,guruswami2022beating,
korner1988new,liu2006explicit,stinson2008some,wang2001explicit,xing2023beating}.

Given a perfect $k$-hash code $C$ of length $n$ we can define the subcode $C_{\neg i:j}:=\{c\in C\mid c_i\neq j\}$ for each integer $1\le i\le n$ and 
$0\le j\le k-1$, i.e., the set of all codewords that do not have symbol $j$ in position $i$. Since removing the $i$th entries from the codewords of 
$C_{\neg i:j}$ gives a perfect $k$-hash code of length $n-1$, we have $\# C\le \left\lfloor \tfrac{k}{k-1} \cdot \# C_{\neg i:j}\right\rfloor $ and    
\begin{equation}
\# C \le (k-1)\cdot \left(\frac{k}{k-1}\right)^n \label{eq_simple}
\end{equation}
for every $k\ge 3$, c.f.\ \cite{korner1988new}.

The simple upper bound (\ref{eq_simple}) was improved by Fredman and Koml{\'o}s back in 1984, see \cite{fredman1984size}, to
\begin{equation}
  \# C \le \left(2^{k!/k^{k-1}}\right)^n
\end{equation}  
for every $k\ge 4$ and sufficiently large $n$. Since then this bound was improved over the years with various restrictions on $k$, see 
e.g.\ \cite{arikan1994upper,costa2021new,della2022improved,guruswami2022beating,korner1988new}. However, no {\lq\lq}proper{\rq\rq} improvement was obtained 
for the case $k=3$. 

Perfect $3$-hash codes are also called \emph{trifferent codes} and we denote the maximum possible cardinality for length 
$n$ by $T(n)$. Clearly the recursion underlying the upper bound (\ref{eq_simple}) can be slightly improved by downrounding to integers in each iteration:
\begin{equation}
  T(n)\le \left\lfloor \frac{3\cdot T(n-1)}{2}\right\rfloor \label{ie_floor}
\end{equation}
for $n\ge 2$. Starting from $T(6)=13$ an iterative application of (\ref{ie_floor}) gives $T(11)\le 94$, so that $T(n)\le 1.0868 \cdot \left(\tfrac{3}{2}\right)^n$ 
for $n\ge 11$, see \cite[Theorem 2.5]{della2022maximum}, while (\ref{eq_simple}) gives $T(n)\le 2\cdot \left(\tfrac{3}{2}\right)^n$.
%% T(7) <= 19 
%% T(8) <= 28
%% T(9) <= 42
%% T(10) <= 63
%% T(11) <= 94

Assuming that the trifferent code $C$ has to be linear, i.e., a linear subspace over $\mathbb{F}_3$ improved upper bounds on $\# C$ can be shown \cite{pohoata2022trifference}. 
Currently the best upper bound is $\# C \le 1.2731^n$ for sufficiently large $n$ \cite{bishnoi2023blocking} (assuming that $C$ is linear). 

Here we determine $T(7)=16$, $T(8)=20$, and $T(9)=27$ -- noting that there exists an up to symmetry unique linear trifferent code of dimension $3$ and length $9$. So, it
remains an interesting open question whether the lower bounds $T(14)\ge 81$ and $T(19)\ge 243$, based on linear codes, from \cite{bishnoi2023blocking} can be beaten.

\section{Results obtained by exhaustive enumeration}
One way to represent trifferent codes is to write the codewords as columns of a matrix. As an example we state the three trifferent codes 
of length $n=6$ and cardinality $T(6)=13$ from \cite{della2022maximum}:
$$
  \left(\begin{smallmatrix}
    0 & 0 & 2 & 2 & 2 & 2 & 2 & 0 & 1 & 1 & 1 & 1 & 1 \\
    0 & 1 & 0 & 2 & 2 & 2 & 1 & 2 & 0 & 1 & 1 & 1 & 2 \\
    0 & 1 & 1 & 0 & 2 & 1 & 2 & 2 & 2 & 0 & 1 & 2 & 1 \\
    0 & 1 & 1 & 1 & 0 & 2 & 2 & 2 & 2 & 2 & 0 & 1 & 1 \\
    0 & 1 & 1 & 2 & 1 & 0 & 2 & 2 & 2 & 1 & 2 & 0 & 1 \\
    0 & 1 & 2 & 1 & 1 & 1 & 0 & 2 & 1 & 2 & 2 & 2 & 0
  \end{smallmatrix}\right)\!\!,
  \left(\begin{smallmatrix}
    0 & 0 & 2 & 2 & 2 & 2 & 2 & 0 & 1 & 1 & 1 & 1 & 1 \\
    0 & 1 & 0 & 2 & 2 & 1 & 2 & 2 & 0 & 1 & 1 & 2 & 1 \\
    0 & 1 & 1 & 0 & 2 & 2 & 1 & 2 & 2 & 0 & 1 & 1 & 2 \\
    0 & 1 & 1 & 1 & 0 & 2 & 2 & 2 & 2 & 2 & 0 & 1 & 1 \\
    0 & 1 & 2 & 1 & 1 & 0 & 2 & 2 & 1 & 2 & 2 & 0 & 1 \\
    0 & 1 & 1 & 2 & 1 & 1 & 0 & 2 & 2 & 1 & 2 & 2 & 0
  \end{smallmatrix}\right)\!\!,
  \left(\begin{smallmatrix}
    0 & 0 & 2 & 2 & 2 & 2 & 2 & 0 & 1 & 1 & 1 & 1 & 1 \\
    0 & 1 & 0 & 2 & 2 & 1 & 2 & 2 & 0 & 1 & 1 & 2 & 1 \\
    0 & 1 & 1 & 0 & 2 & 2 & 2 & 2 & 2 & 0 & 1 & 1 & 1 \\
    0 & 1 & 1 & 1 & 0 & 2 & 1 & 2 & 2 & 2 & 0 & 1 & 2 \\
    0 & 1 & 2 & 1 & 1 & 0 & 2 & 2 & 1 & 2 & 2 & 0 & 1 \\
    0 & 1 & 1 & 1 & 2 & 1 & 0 & 2 & 2 & 2 & 1 & 2 & 0
  \end{smallmatrix}\right)\!\!.
$$
Obviously, applying permutations to the columns, the rows, or the symbols $\{0,1,2\}$ does not change the property of being trifferent, so that we speak of equivalent matrices if they arise
by a sequence of the mentioned operations. With this, two trifferent codes are called equivalent, if their representing matrices are (where the ordering of the columns is arbitrary). Here we are only
interested in non-equivalent codes and choose the lexicographic minimal matrix, when reading columnwise from left to right, as a unique representative of an equivalence class of codes. For our example 
the representatives are given by:
$$
  \left(\begin{smallmatrix}
    0 & 0 & 0 & 0 & 0 & 1 & 1 & 1 & 1 & 1 & 2 & 2 & 2\\
    0 & 0 & 0 & 1 & 2 & 0 & 1 & 1 & 1 & 2 & 0 & 1 & 2\\
    0 & 0 & 0 & 1 & 2 & 1 & 2 & 2 & 2 & 0 & 2 & 0 & 1\\
    0 & 1 & 2 & 0 & 1 & 1 & 0 & 1 & 2 & 0 & 0 & 1 & 2\\
    0 & 1 & 2 & 1 & 2 & 2 & 2 & 0 & 1 & 1 & 1 & 2 & 0\\
    0 & 1 & 2 & 2 & 0 & 0 & 1 & 2 & 0 & 2 & 2 & 0 & 1
  \end{smallmatrix}\right)\!\!,
  \left(\begin{smallmatrix}
    0 & 0 & 0 & 0 & 0 & 1 & 1 & 1 & 2 & 2 & 2 & 2 & 2\\
    0 & 0 & 0 & 1 & 2 & 0 & 1 & 2 & 0 & 1 & 2 & 2 & 2\\
    0 & 0 & 1 & 1 & 2 & 1 & 2 & 0 & 2 & 0 & 0 & 1 & 1\\
    0 & 1 & 2 & 0 & 0 & 1 & 2 & 0 & 1 & 1 & 2 & 0 & 1\\
    0 & 1 & 2 & 1 & 2 & 2 & 0 & 1 & 1 & 2 & 1 & 2 & 0\\
    0 & 1 & 2 & 2 & 0 & 0 & 1 & 2 & 2 & 0 & 0 & 1 & 2
  \end{smallmatrix}\right)\!\!,
  \left(\begin{smallmatrix}
    0 & 0 & 0 & 0 & 0 & 1 & 1 & 1 & 1 & 1 & 2 & 2 & 2\\
    0 & 0 & 0 & 1 & 2 & 0 & 1 & 2 & 2 & 2 & 0 & 1 & 2\\
    0 & 0 & 1 & 0 & 2 & 0 & 2 & 1 & 2 & 2 & 2 & 1 & 0\\
    0 & 1 & 0 & 1 & 2 & 2 & 0 & 1 & 0 & 1 & 1 & 2 & 0\\
    0 & 1 & 1 & 2 & 1 & 2 & 1 & 2 & 2 & 0 & 1 & 0 & 2\\
    0 & 1 & 2 & 0 & 0 & 2 & 2 & 0 & 1 & 2 & 2 & 1 & 0
  \end{smallmatrix}\right)\!\!.
$$
If we interprete every column as the base-3-representation of an integer we get an even more compact representation: 
\begin{itemize}
  \item $\{0,13,26,113,231,285,385,389,399,410,545,582,694\}$, 
  \item $\{0,13,53,113,222,285,397,410,554,582,669,682,686\}$,
  \item $\{0,13,32,96,237,269,383,447,466,470,554,613,654\}$.     
\end{itemize}
%% So, there are at least three non-equivalent trifferent codes of length $6$ and cardinality $13$.

Using the lexicographic ordering we can apply \emph{orderly generation} \cite{read1978every}, i.e., starting from an empty list of codewords we iteratively add 
one codeword at each step such that the code $C$ remains lexicographic minimal and trifferent. In Table~\ref{table_no_equivalent_trifferent_codes_n_le_6} we have 
listed the number $a(n,l)$ of non-equivalent trifferent codes $C$ with length $n$ and cardinality $l$ for all $n\le 6$. Clearly, we have 
$a(n,1)=1$ and $a(n,2)=n$ -- counting the number of possibilities of the Hamming distance $d_H(c,c'):=\#\left\{1\le i\le n \mid c_i\neq c_i'\right\}$ between two 
different codewords of length $n$.\footnote{It is an easy exercise to verify $a(n,3)=n\cdot\left(n^2 + 5\right)/6$.} The count $a(6,13)=3$ was also obtained in \cite{della2022maximum}.

\begin{table}[htp]
  \begin{center}
    \begin{tabular}{rrrrrrrrrrrrrr}
      \hline
      $n \,\backslash\, \# C\!=\!l$ & 1 & 2 & 3 & 4 & 5 & 6 & 7 & 8 & 9 & 10 & 11 & 12 & 13 \\
      \hline
      1 & 1 & 1 &  1 &    \\ 
      2 & 1 & 2 &  3 &  1 \\
      3 & 1 & 3 &  7 &  7 & 2 & 1 \\ 
      4 & 1 & 4 & 14 & 35 & 38 & 25 & 3 & 2 & 1 \\ 
      5 & 1 & 5 & 25 & 141 & 613 & 1410 & 944 & 269 & 55 & 5 \\ 
      6 & 1 & 6 & 41 & 499 & 8038 & 99612 & 486122 & 727128 & 339695 & 63781 & 4832 & 93 & 3 \\
      \hline   
    \end{tabular}  
    \caption{Number $a(n,l)$ of non-equivalent trifferent codes $C$ with length $n$ and cardinality $l$.}
    \label{table_no_equivalent_trifferent_codes_n_le_6}
  \end{center}
\end{table}

\subsection{Upper bounds involving the minimum Hamming distance}
\label{subsec_hamming}
Observing that the three non-equivalent trifferent codes with length $6$ and cardinality $13$ have minimum Hamming distance $d_H(C):=\min\left\{d_H(c,c')\mid c,c'\in C,c\neq c'\right\}=3$
we define $T(n,d)$ as the maximum cardinality of a trifferent code with length $n$ and minimum Hamming distance $d$. From our classification of trifferent codes with length $n\le 6$ 
we conclude the values listed in Table~\ref{table_ub_per_minimum_distance}. We observe that for a relatively large minimum Hamming distance $T(n,d)$ cannot be too large due to 
coding theoretic upper bounds on the maximum cardinality of a $3$-ary code with minimum Hamming distance at least $d$, see e.g.\ \cite{brouwer1998bounds}. Especially, we have $T(n,n)=3$ and 
$T(n,d)=1$ for $d>n$. However, if the minimum Hamming distance $d$ is rather small, then also cardinality $T(n)$ (marked in bold face) cannot be attained.     

\begin{table}[htp]
  \begin{center}
    \begin{tabular}{rrrrrrr}
      \hline
      $n \,\backslash\, d$ & 1 & 2 & 3 & 4 & 5 & 6 \\
      \hline
      1 &  \textbf{3} &  1 &  1 & 1 & 1 & 1 \\ 
      2 &  \textbf{4} &  3 &  1 & 1 & 1 & 1 \\
      3 &  5 &  \textbf{6} &  3 & 1 & 1 & 1 \\ 
      4 &  7 &  8 &  \textbf{9} & 3 & 1 & 1 \\ 
      5 & \textbf{10} & \textbf{10} &  9 & 6 & 3 & 1 \\ 
      6 & 11 & 12 & \textbf{13} & 10 & 4 & 3 \\
      \hline   
    \end{tabular}  
    \caption{Maximum cardinality $T(n,d)$ of a trifferent code with length $n$ and minimum Hamming distance $d$.}
    \label{table_ub_per_minimum_distance}
  \end{center}
\end{table}

Our next aim is to determine $T(n,1)$.
\begin{lemma}
  \label{lemma_hamming_distance_1}
  Let $C$ be a trifferent code with length $n$ and $c\in C$ be arbitrary. Then, 
  $$
    C':=\left\{(c,0),(c,1)\right\} \cup \left\{(c',2)\mid c'\in C\backslash \{c\}\right\}
  $$
  is a trifferent code with length $n+1$, minimum Hamming distance $1$, and cardinality $\# C+1$.
\end{lemma}
\begin{proof}
  Let $x',y',z'$ be three different elements from $C'$ and $x,y,z$ be the corresponding codewords after removing the last coordinate. 
  If the last coordinate of at most one element of $x',y',z'$ is  not equal to $2$, then $x,y,z$ are different elements in $C$ and $\{x,y,z\}$ 
  is trifferent, by assumption, and so is $\{x',y',z'\}$. If the last coordinate of at least two elements of $x',y',z'$ is  not equal to $2$, then 
  $x',y',z'$ differ in the last coordinate. Thus, $C'$ is trifferent. The other two statements on the cardinality and minimum Hamming distance of $C'$ 
  directly follow from the construction. 
\end{proof}

\begin{corollary}
  For each $n\ge 2$ we have $T(n)\ge T(n,1)\ge T(n-1)+1$.
\end{corollary}

\begin{theorem}
  \label{thm_hamming_distance_1}
  For each $n\ge 2$ we have $T(n,1)=T(n-1)+1$.  
\end{theorem}
\begin{proof}
  It remains to show $\#C\le T(n-1)+1$ for each trifferent code with length $n$ and minimum Hamming distance $1$. W.l.o.g.\ we assume 
  $c':=(0,\dots,0)\in C$ and $c'':=(1,0,\dots,0)\in C$, so that every other codeword $c$ has a $2$ in the first coordinate (since $\{c',c'',c\}$ is trifferent). 
  Thus, we have
  \begin{equation}
    \# C \le \# C_{\neg 1:0}+1\le T(n-1)+1.
  \end{equation}
\end{proof}

\begin{lemma} 
  \label{lemma_not_surjective}
  Let $d\ge 2$ and let $\pi\colon C\to \{0,1,2\}^d$ be the projection on the first $d$ coordinates. Suppose that $\pi$ is not surjective. Then $\#C \le T(n-d)\cdot \left(3^d-1\right)/2^d$.
\end{lemma}  
\begin{proof} 
W.l.o.g.\ the vector $(1,\dots,1)$ is not in the image of $\pi$. We call $z\in \{0,1,2\}^d$ \emph{odd} if $z_1+\dots+z_d$ is odd and \emph{even} if this sum is even. For 
$z$ in  $\{0,1,2\}^d$ define $B_z:=\left\{w\in \{0,1,2\}^d\,:\, w_i\neq z_i \,\forall i=1,\dots,d\right\}$   and   $C_z:=\left\{c\in C\,:\, \pi(c)\in B_z\right\}$. 
Observe that for any $w\in \{0,1,2\}^d$ unequal to $(1,\dots,1)$, there are $2^{d-1}$ sets $B_z$ that contain $w$ for which $z$ is odd (and equally many for $z$ even).  
Since $\#C_z \le T(n-d)$ for all $z$, we obtain 
$$
  2^{d-1}\cdot\#C= \sum_{z \text{ odd}} \#C_z \leq T(n-d) \cdot \left(3^d-1\right)/2
$$
since there are $\left(3^d-1\right)/2$ vectors $z$ that are odd.
\end{proof}

\begin{corollary} 
  \label{cor_ub_hamming}
  If a trifferent code $C$  of length $n$ has a pair at Hamming distance $d \ge 2$, then we have $\#C \le T(n-d)\cdot\left(3^d-1\right)/2^d$.
\end{corollary}
\begin{proof}  
  W.l.o.g.\ the two words at distance $d$ are $0\dots00\dots0$ and $1\dots10\dots0$ in $\{0,1,2\}^{d+(n-d)}$. Then every other codeword $c\in C$ must have a $2$ in the first $d$ 
  positions because of the trifferent property. So $\pi$ is not surjective and we can apply Lemma~\ref{lemma_not_surjective}.
\end{proof}

Since the convergence of $\tfrac{3^d-1}{2^d}$ towards $\tfrac{3^d}{2^d}$ is too fast we even cannot deduce an improved upper bound of the form $T(n)\le c\cdot \left(\tfrac{3}{2}\right)^n/\log\log n$ 
for some constant $c>0$ and sufficiently large $n$.  

\subsection{Upper bounds involving the number of occurrences of symbols}
For a given trifferent code $C$ with length $n\ge 2$ let $s_{i,j}$ be the number of codewords $c\in C$ with $c_i=j$, where $1\le i\le n$ and $j\in\{0,1,2\}$. For each triple of pairwise different symbols 
$j_1,j_2,j_3\in\{0,1,2\}$ we have $s_{i,j_1}+s_{i,j_2}=\#C_{\neg i: j_3}\le T(n-1)$, so that $\# C=s_{i,j_1}+s_{i,j_2}+s_{i,j_3}$ implies 
\begin{equation} 
  \# C\le 2\cdot T(n-1)-s_{i,j_1}. \label{ie_max_occurence}
\end{equation}
Let $\tau=\tau(C)$ denote the maximum value of the sums $s_{i,j_1}+s_{i,j_2}$ and $i^\star\!,j_1^\star,j_2^\star,j_3^\star$ be such that $\tau=s_{i^\star\!,j_1^\star}+s_{i^\star\!,j_2^\star}$, 
$\left\{j_1^\star,j_2^\star,j_3^\star\right\}=\{0,1,2\}$. With this, let $C'$ arise from $C_{\neg i^\star:j_3^\star}$ by removing the $i^\star$th coordinates from the codewords, i.e., 
$C'$ is a trifferent code with length $n-1$ and cardinality $\tau$.  
 
As the numbers $a(n,l)$ of non-equivalent trifferent codes with length $n$ and cardinality $l$ grow very quickly, see Table~\ref{table_no_equivalent_trifferent_codes_n_le_6}, 
we want restrict our exhaustive search to rather large values of $l$. To this end we start from a trifferent code $C'$ with length $n-1$ and cardinality $\tau$, fill the $n$th coordinate with all $2^\tau$ 
choices of symbols in $\{0,1\}$, and iteratively extend by codewords have a $2$ in the $n$th coordinate such that the code remains trifferent. Starting from an arbitrary trifferent code $C$ of length $n$ to
choose $C'$ as described above, the list of constructed trifferent codes of length $n$ contains a code that is equivalent to $C$. Assuming that $\# C'=\tau(C)$ we can upper bound the cardinality 
of $C$ by 
\begin{equation}
  \#C =\frac{1}{2}\cdot \left(s_{1,0}+s_{1,1}\,\,+\,\,s_{1,0}+s_{1,2}\,\,+\,\,s_{1,1}+s_{1,2}\right)\le \left\lfloor\frac{3\tau}{2}\right\rfloor = \left\lfloor\frac{3\# C'}{2}\right\rfloor.
\end{equation}
So, if we want to construct all trifferent codes with length $9$ and cardinality at least $28$, then it suffices to start with all trifferent codes of length $8$ and cardinality at least $19$. The latter 
can be constructed from trifferent codes with length $7$ and cardinality at least $13$ and those can in turn be constructed from all trifferent codes of length $6$ and cardinality at least $9$.

Given the parameter $\tau=\# C'$ we can slightly improve Inequality~(\ref{ie_max_occurence}) to
\begin{equation}
  \# C \le 2\tau -s_{i,j_1}.
\end{equation}
While we do not know the exact value of $s_{i,j_1}$ in the intermediate steps of the generation of the codes $C$, we always have lower bounds for $s_{i,j_1}$ which we use to cut the search trees. 
As computational results we have obtained
\begin{itemize}
  \item $a(7,13)=22736583$, $a(7,14)=429602$, $a(7,15)=2164$, $a(7,16)=3$, $a(7,17)=0$; 
  \item $a(8,19)=38581$, $a(8,20)=57$, $a(8,21)=0$;
  \item $a(9,28)=0$,
\end{itemize}  
which implies:
\begin{theorem}
  $T(7)=16$, $T(8)=20$, $T(9)\le 27$, and $T(n)\le 0.6937\cdot \left(\tfrac{3}{2}\right)^n$ for $n\ge 10$.
\end{theorem}
%% --------------------------------------------------------------------------------------------------------------------
%% 4928 codes read.
%% 1004966 with size at least 14 written
%% 
%% real    143m52,243s
%% user    143m22,705s
%% sys    0m11,954s
%% btmw06@btmwxu:~$ grep "15 0 " codes.txt | wc -> 9237  147792  609121
%% btmw06@btmwxu:~$ grep "16 0 " codes.txt | wc -> 14     238    1007
%% btmw06@btmwxu:~$ grep "14 0 " codes.txt | wc -> 995715 14935725 60983626
%% btmw06@btmwxu:~$ grep "14 0 " orb7.txt | wc -> 385847 5787705 23686164
%% btmw06@btmwxu:~$ grep "15 0 " orb7.txt | wc -> 2151   34416  141884
%% btmw06@btmwxu:~$ grep "16 0 " orb7.txt | wc -> 3      51     216
%% --------------------------------------------------------------------------------------------------------------------
%% 431769 codes read.
%% 288233 with size at least 18 written
%% 
%% real    2214m2,134s
%% user    2213m11,834s
%% sys     0m1,396s
%% sascha@sascha-ThinkPad-P15-Gen-1:~/Downloads$ grep "18 0 " orb8.txt | wc -> 257367 4889973 21619903
%% sascha@sascha-ThinkPad-P15-Gen-1:~/Downloads$ grep "19 0 " orb8.txt | wc -> 2996   59920  267570
%% sascha@sascha-ThinkPad-P15-Gen-1:~/Downloads$ grep "20 0 " orb8.txt | wc -> 41     861    3934
%% --------------------------------------------------------------------------------------------------------------------

We remark that extending the $38638$ non-equivalent trifferent codes with length $8$ and cardinality at least $19$ yields just $44$ non-equivalent trifferent codes with length $9$, cardinality $25$ and 
none with cardinality $26$. However, one can easily construct a linear trifferent code of length $9$ and cardinality $27$, as we will explain in the next section, so that $T(9)=27$:
$$
C=\left(\begin{smallmatrix}
0 & 0 & 0 & 0 & 0 & 0 & 0 & 0 & 0 & 1 & 1 & 1 & 1 & 1 & 1 & 1 & 1 & 1 & 2 & 2 & 2 & 2 & 2 & 2 & 2 & 2 & 2\\
0 & 0 & 0 & 1 & 1 & 1 & 2 & 2 & 2 & 0 & 0 & 0 & 1 & 1 & 1 & 2 & 2 & 2 & 0 & 0 & 0 & 1 & 1 & 1 & 2 & 2 & 2\\
0 & 0 & 0 & 1 & 1 & 1 & 2 & 2 & 2 & 1 & 1 & 1 & 2 & 2 & 2 & 0 & 0 & 0 & 2 & 2 & 2 & 0 & 0 & 0 & 1 & 1 & 1\\
0 & 0 & 0 & 1 & 1 & 1 & 2 & 2 & 2 & 2 & 2 & 2 & 0 & 0 & 0 & 1 & 1 & 1 & 1 & 1 & 1 & 2 & 2 & 2 & 0 & 0 & 0\\
0 & 1 & 2 & 0 & 1 & 2 & 0 & 1 & 2 & 0 & 1 & 2 & 0 & 1 & 2 & 0 & 1 & 2 & 0 & 1 & 2 & 0 & 1 & 2 & 0 & 1 & 2\\
0 & 1 & 2 & 0 & 1 & 2 & 0 & 1 & 2 & 1 & 2 & 0 & 1 & 2 & 0 & 1 & 2 & 0 & 2 & 0 & 1 & 2 & 0 & 1 & 2 & 0 & 1\\
0 & 1 & 2 & 0 & 1 & 2 & 0 & 1 & 2 & 2 & 0 & 1 & 2 & 0 & 1 & 2 & 0 & 1 & 1 & 2 & 0 & 1 & 2 & 0 & 1 & 2 & 0\\
0 & 1 & 2 & 1 & 2 & 0 & 2 & 0 & 1 & 0 & 1 & 2 & 1 & 2 & 0 & 2 & 0 & 1 & 0 & 1 & 2 & 1 & 2 & 0 & 2 & 0 & 1\\
0 & 1 & 2 & 2 & 0 & 1 & 1 & 2 & 0 & 0 & 1 & 2 & 2 & 0 & 1 & 1 & 2 & 0 & 0 & 1 & 2 & 2 & 0 & 1 & 1 & 2 & 0
\end{smallmatrix}\right)
$$
We can computationally check that $C_{\neg i:j}$ has cardinality $18$ for all $1\le i\le 9$, $j\in\{0,1,2\}$ (which can be deduced from the 
fact that $C$ is linear with dual minimum distance of at least $2$). Repeating the process a second time we can computationally check 
that all resulting trifferent codes of length $7$ have cardinality $12$ (which can also be deduced from the 
fact that $C$ is linear with dual minimum distance of at least $3$, i.e., that the code is projective). So, clearly the search described above could not find a 
code equivalent to $C$. However, we can apply the extension search to all $C_{\neg i:j}$ after removing the $i$th coordinates. As a result we have obtained 
$166$, $39$, and $11$ non-equivalent trifferent codes with lengths $9$ and cardinalities $25$, $26$, and $27$, respectively.

\section{Linear trifferent codes}
\label{sec_linear}

We call a $q$-ary code $C$ that forms a linear subspace over $\mathbb{F}_q$ a \emph{linear} code. If $C$ has length $n$ and cardinality $q^k$, then 
we speak of an $[n,k]_q$-code. A non-zero codeword $c\in C$ is called \emph{minimal} if the \emph{support} $\supp(c):=\{i\mid c_i\neq 0\}$ of $c$ is minimal with respect to 
inclusion in the set $\left\{\supp(u) \mid u\in C\backslash\mathbf{0}\right\}$. The code $C$ is a \emph{minimal code} if all its non-zero codewords are minimal. 
In \cite{bishnoi2023blocking} it was shown that a linear $3$-ary code is minimal iff it is trifferent.

In contrast to Theorem~\ref{thm_hamming_distance_1} and Lemma~\ref{thm_hamming_distance_1} the minimum Hamming distance of a linear trifferent code needs to increase 
with the length of the codewords, i.e., for every minimal $[n,k]_q$-code we have $d\ge (k-1)(q-1)+1$ for the minimum Hamming distance $d$, see e.g.\ \cite[Theorem 2.8]{three_combinatorial_perspectives}, \cite[Theorem 23]{heger2021short}.  
Using the software \texttt{LinCode} for the enumeration of linear codes \cite{bouyukliev2021computer} we have classified the minimal $[n,k]_3$-codes with the minimum possible length $n$
for small $k$ up to symmetry. For dimension $k=1$ the smallest example is the trivial $[1,1]_3$ code with minimum distance $1$ and for $k=2$ the unique smallest example is a $[4,2]_3$ 
code with minimum distance $3$. Also for dimension $k=3$ there exists a unique smallest example\footnote{We remark that there are nine $[9,3]_3$-codes and one $[8,3]_3$-code 
with minimum distance at least $5$ -- only one of these codes is minimal.} given e.g.\ by the generator matrix
$$
  \begin{pmatrix}
  0 & 0 & 0 & 0 & 1 & 1 & 1 & 1 & 1 \\
  0 & 1 & 1 & 1 & 0 & 0 & 0 & 1 & 2 \\
  1 & 0 & 1 & 2 & 0 & 1 & 2 & 2 & 2
  \end{pmatrix}\!\!.
$$ 
The code has weight enumerator $1+6x^{5}+8x^{6}+12x^{7}$, i.e., its minimum distance is $5$, and its automorphism group has order $48$. All of these codes attain the maximum cardinality 
$T(n)=3^k$ of trifferent codes for their respective length $n$ and for $k=1,2$ even all such examples are linear while there are others for $k=3$. For dimension $k=4$ the minimum length
$n$ of a minimal $[n,k]_3$-code is $14$ and there are exactly three non-equivalent such codes given e.g.\ by the generator matrices
$$
  \begin{pmatrix}
    11111111101000 \\
    00011122210100 \\
    11101201220010 \\
    01221211200001
  \end{pmatrix}\!\!,
  \begin{pmatrix}
    11111111101000 \\
    00011112210100 \\
    01211121220010 \\
    10001220110001
  \end{pmatrix}\!\!,\text{ and }
  \begin{pmatrix}
    11111111101000 \\
    00011112210100 \\
    11202221220010 \\
    02010121020001
  \end{pmatrix}\!\!.
$$  
The corresponding weight enumerators are given by $1+14x^{7}+4x^{8}+20x^{9}+16x^{10}+26x^{11}$,  
$1+12x^{7}+12x^{8}+8x^{9}+24x^{10}+24x^{11}$, $1+12x^{7}+12x^{8}+8x^{9}+24x^{10}+24x^{11}$  and 
the orders of the corresponding automorphism groups are $16$, $12$, $48$, respectively. For larger 
values of $n$ the number of non-equivalent minimal $[n,4]_3$-codes grows quickly, i.e., there are $53$ 
minimal $[15,4]_3$-, $818$ minimal $[16,4]_3$-, $9266$ minimal $[17,4]_3$-, and $80999$ minimal $[18,4]_3$-codes. 
We remark that the three minimal $[14,4]_3$-codes all violate the \emph{Ashikhmin--Barg condition} $w_{\max}<w_{\min} \cdot\tfrac{q}{q-1}$ 
which is a sufficient condition on the maximum and minimum weight of a linear $q$-ary code to be minimal \cite{ashikhmin1998minimal}.   

For dimension $k = 5$ the minimum length $n$ of a minimal $[n,k]_3$-code is $19$ and up to equivalence there is a unique such code given e.g. by the generator matrix
$$
  \begin{pmatrix}
    1111111111100010000\\
    0001111222211101000\\
    1110111022201200100\\
    0122000122220000010\\
    1220012101212100001
  \end{pmatrix}\!\!.
$$
The corresponding weight enumerator is given by $1+26x^{9}+132x^{12}+84x^{15}$, i.e., the code is a $3$-divisible $3$-weight code, and the automorphism group has order $8$. 
(The support contains two full lines intersecting in a common point.) One dimension less,
there are $27282$ minimal $[18,4,9]_3$-codes. For the next case we report $263$ minimal $[19,4,11]_3$-, $49104$ minimal $[20,4,11]_3$-, and $1615534$ minimal $[21,4,11]_3$ codes. 
A minimal $[n,5,11]_3$-code with maximum column multiplicity has length at least $21$ and we have found $101$ non-isomorphic such $[21,5,11]_3$-codes.
None of them can be extended to a minimal $[22,6,11]_3$-code, so that $m(6,3)\ge 23$. Extending roughly 2\% of the $1615534$ minimal $[21,4,11]_3$ codes, we constructed
$2165971$ minimal $[22,5,11]_3$ codes. Extending $6784$ of the latter unfortunately has not resulted in a minimal $[23,6,11]_3$ code, while one might conjecture that such a code exist.   
We remark that the minimum lengths of minimal $[n,k]_3$-codes for $k\le 5$ have also been determined in \cite{bishnoi2023blocking}. 
Using integer linear programming computations it was also shown that a minimal $[n,6]_3$-code has length at least $22$ and an example with length $24$ was found. 
%% suche nach minimalen [23,6,11]_3; [22,5,11]_3; [21,4,11]_3, [20,5,11]_3 codes  
%% läuft in dev_new2

\medskip

To conclude this section we consider the well-known correspondence between (non-degenerated) $[n,k]_q$-codes and 
multisets of points in the projective space $\operatorname{PG}(k-1,q)$ of cardinality $n$, i.e., the columns of a 
generator matrix each generate a point, see e.g.\ \cite{dodunekov1998codes}.
\begin{definition}
  A multiset $M$ of points in a projective space is called a \emph{strong blocking multiset} if for every
  hyperplane $H$, we have $\langle S \cap H\rangle = H$.
\end{definition}
If $M$ is the multiset of points associated to a linear code $C$, then $C$ is minimal iff $M$ is a strong blocking multiset, see e.g.\ \cite{alfarano2022geometric,tang2021full}. 
Directly from the definition of a strong blocking multiset we can read off that a multiset of points in $\operatorname{PG(1,q)}$ is a strong blocking multiset 
iff it contains every point of the entire projective space. Clearly adding points to a multiset does not destroy the property of being a strong blocking multiset, 
so that we consider \emph{minimal strong blocking sets} in the following, i.e., set of points that are a strong blocking multiset but such that every proper 
subset is not a strong blocking multiset. So, in $\PG(1,q)$ the unique minimal strong blocking set is a line.\footnote{For completeness we remark that the corresponding 
minimal $[4,2]_3$-code has an automorphism group of order $48$.} As e.g.\ observed in \cite{alfarano2022geometric} 
a set of points in $\operatorname{PG(2,q})$ is a strong blocking set iff every line contains at least $2$ points, i.e., a $2$-fold blocking set (with respect to lines).
In $\operatorname{PG(2,3)}$ the complement of a minimal strong blocking set is an \emph{oval}, which, for $q=3$, coincides with a (projective) frame, i.e., a set of 
$4$ points each three of them spanning the entire space.\footnote{We said that it is {\lq\lq}easy{\rq\rq} to construct a minimal (or trifferent) $[9,3]_3$-code, since
we can start with the point set of the projective plane $\operatorname{PG(2,3)}$ and iteratively remove arbitrary points that are not contained on any secant. This process 
always stops after $4$ removals and we end up with a unique set of points (up to equivalence).} While there is an unique minimal strong blocking set in $\operatorname{PG}(2,3)$,
in $\operatorname{PG}(3,3)$ we have several non-equivalent such sets, see Table~\ref{table_minimal_strong_blocking_sets}. So, if we start from the set of $40$ points of the 
projective solid $\operatorname{PG}(3,3)$ and iteratively remove points without destroying the property of being a strong blocking set, we can delete between $26$ and $20$ points.

\begin{table}[htp]
  \begin{center}
    \begin{tabular}{lrrrrrrr}
      \hline
      $\mathbf{n}$ & 14 & 15 & 16 & 17 & 18 & 19 & 20 \\ 
      $\mathbf{\#}$ & 3 & 39 & 363 & 1517 & 3736 & 5791 & 5110 \\     
      \hline
    \end{tabular}
    \caption{Number of non-equivalent minimal strong blocking sets in $\operatorname{PG}(3,3)$ per cardinality $n$.}
    \label{table_minimal_strong_blocking_sets}  
  \end{center}
\end{table}
    
\section*{Acknowledgments}
The author thanks Gianira Alfarano, Anurag Bishnoi, Jozefien D'haeseleer, Dion Gijswijt, Alessandro Neri, Sven Polak, and Martin Scotti for many helpful remarks on an earlier version of this paper. 
Especially, Lemma~\ref{lemma_not_surjective} and Corollary~\ref{cor_ub_hamming} are stated with courtesy of Dion Gijswijt (simplifying and generalizing partial results of the author).     

%% \section*{Declaration of competing interest}
%% The author declares that he has no known competing financial interest or personal relationship that could have appeared to
%% influence the work reported in this paper.

\appendix

\section{Trifferent codes with small lengths and relatively large cardinalities}  

\setitemize{leftmargin=*}            
There exist $a(4,7)=3$  non-equivalent trifferent codes with length $4$ and cardinality $7$: 
{\tiny
\begin{itemize}
  \item $\{0,1,14,35,50,71,74\}$;\\[-4mm]
  \item $\{0,13,26,32,42,46,61\}$;\\[-4mm]
  \item $\{0,4,11,34,47,70,77\}$.
\end{itemize}
}
\noindent
There exist $a(4,8)=2$  non-equivalent trifferent codes with length $4$ and cardinality $8$: 
{\tiny
\begin{itemize}
  \item $\{0,13,26,32,42,46,61,65\}$;\\[-4mm]
  \item $\{0,4,11,34,47,70,77,78\}$.
\end{itemize}
}
\noindent
There exists an up to equivalence unique trifferent code with length $4$ and maximum cardinality $T(4)=9$:
{\tiny
\begin{itemize}
  \item $\{0,13,26,32,42,46,61,65,75\}$.
\end{itemize}
}
\noindent
The code is linear.

\bigskip
\bigskip    

There exist $a(5,10)=5$ non-equivalent trifferent codes with length $5$ and maximum cardinality $T(5)=10$:
{\tiny
\begin{itemize}
  \item $\{0,1,41,80,98,128,140,185,197,227\}$;\\[-4mm]
  \item $\{0,4,11,34,101,142,209,232,239,240\}$;\\[-4mm]
  \item $\{0,4,38,79,97,131,132,137,182,196\}$;\\[-4mm]
  \item $\{0,4,38,79,97,131,137,182,196,230\}$;\\[-4mm]
  \item $\{0,4,38,79,97,131,137,182,196,231\}$.
\end{itemize}
}

\bigskip
\bigskip

There exist $a(6,12)=93$ non-equivalent trifferent codes with length $6$ and cardinality $12$:
{\tiny
\begin{itemize}
  \item $\{0,13,113,123,223,227,289,399,431,566,586,696\}$;\\[-4mm]
  \item $\{0,13,26,113,123,223,289,308,399,586,659,696\}$;\\[-4mm]
  \item $\{0,13,26,113,231,285,385,389,399,410,545,582\}$;\\[-4mm]
  \item $\{0,13,26,113,231,285,385,389,399,410,545,586\}$;\\[-4mm]
  \item $\{0,13,26,113,231,285,385,389,399,410,545,694\}$;\\[-4mm]
  \item $\{0,13,26,113,231,285,385,389,410,545,586,696\}$;\\[-4mm]
  \item $\{0,13,26,32,123,289,385,389,399,420,586,707\}$;\\[-4mm]
  \item $\{0,13,26,32,123,289,466,470,480,640,663,707\}$;\\[-4mm]
  \item $\{0,13,26,32,96,223,227,237,289,447,586,626\}$;\\[-4mm]
  \item $\{0,13,26,32,96,223,227,289,447,480,586,626\}$;\\[-4mm]
  \item $\{0,13,26,32,96,223,289,447,470,480,586,626\}$;\\[-4mm]
  \item $\{0,13,26,32,96,223,289,447,470,586,626,723\}$;\\[-4mm]
  \item $\{0,13,26,32,96,223,289,447,480,586,626,713\}$;\\[-4mm]
  \item $\{0,13,26,32,96,262,383,447,613,709,713,723\}$;\\[-4mm]
  \item $\{0,13,26,32,96,262,464,640,690,709,713,723\}$;\\[-4mm]
  \item $\{0,13,26,32,96,289,447,466,470,480,586,626\}$;\\[-4mm]
  \item $\{0,13,26,32,96,289,447,466,470,480,586,707\}$;\\[-4mm]
  \item $\{0,13,26,32,96,289,447,466,470,586,626,723\}$;\\[-4mm]
  \item $\{0,13,32,107,123,289,304,389,399,586,654,697\}$;\\[-4mm]
  \item $\{0,13,32,107,123,304,399,461,586,654,697,716\}$;\\[-4mm]
  \item $\{0,13,32,107,232,285,308,385,399,582,626,694\}$;\\[-4mm]
  \item $\{0,13,32,96,127,237,269,466,470,554,620,654\}$;\\[-4mm]
  \item $\{0,13,32,96,134,153,311,457,480,613,670,707\}$;\\[-4mm]
  \item $\{0,13,32,96,134,262,461,475,554,613,654,723\}$;\\[-4mm]
  \item $\{0,13,32,96,223,269,289,447,635,639,670,704\}$;\\[-4mm]
  \item $\{0,13,32,96,223,269,289,480,635,670,681,704\}$;\\[-4mm]
  \item $\{0,13,32,96,223,269,370,480,635,670,681,704\}$;\\[-4mm]
  \item $\{0,13,32,96,223,269,383,447,470,480,613,670\}$;\\[-4mm]
  \item $\{0,13,32,96,223,269,451,480,635,670,681,704\}$;\\[-4mm]
  \item $\{0,13,32,96,237,269,289,447,466,641,654,697\}$;\\[-4mm]
  \item $\{0,13,32,96,237,269,383,447,466,470,554,613\}$;\\[-4mm]
  \item $\{0,13,32,96,237,269,383,447,466,470,554,654\}$;\\[-4mm]
  \item $\{0,13,32,96,237,311,370,411,626,690,709,713\}$;\\[-4mm]
  \item $\{0,13,53,113,222,285,397,410,554,582,669,682\}$;\\[-4mm]
  \item $\{0,13,53,113,222,285,410,554,582,669,682,686\}$;\\[-4mm]
  \item $\{0,4,11,115,236,290,394,401,402,412,547,587\}$;\\[-4mm]
  \item $\{0,4,11,115,236,290,394,401,412,547,587,699\}$;\\[-4mm]
  \item $\{0,4,11,115,236,290,394,402,412,547,587,698\}$;\\[-4mm]
  \item $\{0,4,11,124,239,293,385,398,421,556,590,699\}$;\\[-4mm]
  \item $\{0,4,11,124,240,294,385,398,421,556,591,698\}$;\\[-4mm]
  \item $\{0,4,11,34,101,304,425,628,695,718,725,726\}$;\\[-4mm]
  \item $\{0,4,119,239,290,380,416,483,554,590,610,680\}$;\\[-4mm]
  \item $\{0,4,119,240,290,380,416,482,555,591,610,681\}$;\\[-4mm]
  \item $\{0,4,38,101,149,218,322,358,421,455,636,695\}$;\\[-4mm]
  \item $\{0,4,38,101,149,322,358,393,421,461,699,707\}$;\\[-4mm]
  \item $\{0,4,38,101,150,218,322,358,421,456,635,695\}$;\\[-4mm]
  \item $\{0,4,38,101,150,322,358,392,421,461,698,708\}$;\\[-4mm]
  \item $\{0,4,38,104,286,294,384,392,455,551,655,668\}$;\\[-4mm]
  \item $\{0,4,38,104,286,303,311,428,591,632,655,695\}$;\\[-4mm]
  \item $\{0,4,38,124,241,293,383,421,461,555,591,682\}$;\\[-4mm]
  \item $\{0,4,38,124,241,293,383,421,554,591,668,682\}$;\\[-4mm]
  \item $\{0,4,38,124,241,293,383,421,555,591,668,682\}$;\\[-4mm]
  \item $\{0,4,38,124,241,293,384,421,461,555,590,682\}$;\\[-4mm]
  \item $\{0,4,38,124,241,293,384,421,555,590,668,682\}$;\\[-4mm]
  \item $\{0,4,38,124,241,294,383,421,461,555,591,682\}$;\\[-4mm]
  \item $\{0,4,38,124,241,294,383,421,555,591,668,682\}$;\\[-4mm]
  \item $\{0,4,38,50,97,205,267,464,474,537,623,646\}$;\\[-4mm]
  \item $\{0,4,38,50,97,231,267,448,464,565,618,623\}$;\\[-4mm]
  \item $\{0,4,38,51,124,241,293,421,461,554,590,627\}$;\\[-4mm]
  \item $\{0,4,38,51,97,205,266,465,473,536,623,646\}$;\\[-4mm]
  \item $\{0,4,38,51,97,230,266,448,465,565,617,623\}$;\\[-4mm]
  \item $\{0,4,38,79,131,230,294,340,380,555,668,682\}$;\\[-4mm]
  \item $\{0,4,38,79,131,294,340,380,473,636,668,682\}$;\\[-4mm]
  \item $\{0,4,38,79,131,312,425,439,473,583,618,623\}$;\\[-4mm]
  \item $\{0,4,38,79,131,312,425,439,473,583,618,704\}$;\\[-4mm]
  \item $\{0,4,38,79,131,312,425,439,473,618,623,664\}$;\\[-4mm]
  \item $\{0,4,38,79,131,312,425,439,473,618,664,704\}$;\\[-4mm]
  \item $\{0,4,38,79,131,312,425,439,473,636,664,704\}$;\\[-4mm]
  \item $\{0,4,38,79,97,131,137,294,473,618,668,682\}$;\\[-4mm]
  \item $\{0,4,38,79,97,131,299,456,623,668,682,716\}$;\\[-4mm]
  \item $\{0,4,38,79,97,131,299,456,623,668,682,717\}$;\\[-4mm]
  \item $\{0,4,38,79,97,132,299,455,623,668,682,716\}$;\\[-4mm]
  \item $\{0,4,38,79,97,132,299,455,623,668,682,717\}$;\\[-4mm]
  \item $\{0,4,38,79,97,137,230,294,374,668,682,717\}$;\\[-4mm]
  \item $\{0,4,38,79,97,230,294,374,380,668,682,717\}$;\\[-4mm]
  \item $\{0,4,38,79,97,230,294,374,623,668,682,717\}$;\\[-4mm]
  \item $\{0,4,38,79,97,231,293,375,380,668,682,716\}$;\\[-4mm]
  \item $\{0,4,38,79,97,231,293,375,623,668,682,716\}$;\\[-4mm]
  \item $\{0,4,38,79,97,293,421,618,623,668,682,716\}$;\\[-4mm]
  \item $\{0,4,38,79,97,294,421,617,623,668,682,716\}$;\\[-4mm]
  \item $\{0,4,38,79,97,299,421,623,668,682,716,717\}$;\\[-4mm]
  \item $\{0,4,38,97,131,267,473,542,618,646,655,695\}$;\\[-4mm]
  \item $\{0,4,38,97,131,267,473,565,618,623,668,682\}$;\\[-4mm]
  \item $\{0,4,38,97,131,294,299,403,591,668,682,716\}$;\\[-4mm]
  \item $\{0,4,38,97,131,294,322,380,473,591,655,695\}$;\\[-4mm]
  \item $\{0,4,38,97,146,221,267,455,474,591,682,695\}$;\\[-4mm]
  \item $\{0,4,38,97,146,222,266,456,473,590,682,695\}$;\\[-4mm]
  \item $\{0,4,38,97,159,266,448,465,473,590,655,695\}$;\\[-4mm]
  \item $\{0,4,38,97,230,266,375,565,623,668,682,717\}$;\\[-4mm]
  \item $\{0,4,38,97,230,266,389,448,465,565,617,668\}$;\\[-4mm]
  \item $\{0,4,38,97,230,266,389,448,465,565,618,668\}$;\\[-4mm]
  \item $\{0,4,38,97,266,308,421,618,623,682,707,717\}$;\\[-4mm]
  \item $\{0,4,38,97,294,299,421,590,632,682,707,717\}$.
\end{itemize}
}
\noindent
There exist $a(6,13)=3$ non-equivalent trifferent codes with length $6$ and maximum cardinality $T(6)=13$:
{\tiny
\begin{itemize}
  \item $\{0,13,26,113,231,285,385,389,399,410,545,582,694\}$;\\[-4mm]
  \item $\{0,13,32,96,237,269,383,447,466,470,554,613,654\}$;\\[-4mm]
  \item $\{0,13,53,113,222,285,397,410,554,582,669,682,686\}$.
\end{itemize}
}

\bigskip
\bigskip

There exist $a(7,16)=3$ non-equivalent trifferent codes with length $7$ and maximum cardinality $T(7)=16$:
{\tiny
\begin{itemize}
  \item $\{0,13,110,277,426,458,797,885,929,1258,1342,1455,1804,1841,1934,2004\}$;\\[-4mm]
  \item $\{0,13,110,277,610,662,798,887,1202,1419,1599,1723,1865,1924,2101,2124\}$;\\[-4mm]
  \item $\{0,13,110,277,663,801,878,1021,1203,1339,1424,1561,1761,2032,2081,2124\}$.
\end{itemize}
}

\bigskip
\bigskip

There exist $a(8,20)=57$ non-equivalent trifferent codes with length $8$ and maximum cardinality $T(8)=20$:
{\tiny
\begin{itemize}
  \item $\{0,13,110,1006,1202,2016,2074,2508,2588,3418,3539,3760,3821,3849,5074,5087,5206,5340,6094,6293\}$;\\[-4mm]
  \item $\{0,13,110,1006,1202,2016,2074,2508,2588,3418,3539,3760,3821,3849,5074,5087,5206,5343,6094,6293\}$;\\[-4mm]
  \item $\{0,13,110,1006,1202,2016,2074,2508,2588,3418,3539,3859,3872,4975,5036,5064,5206,5337,6094,6293\}$;\\[-4mm]
  \item $\{0,13,110,1006,1202,2016,2074,2508,2588,3418,3539,3859,3872,4975,5036,5064,5206,5340,6094,6293\}$;\\[-4mm]
  \item $\{0,13,110,1006,1202,2016,2074,2508,2588,3418,3539,3859,3872,4975,5036,5064,5206,5343,6094,6293\}$;\\[-4mm]
  \item $\{0,13,110,1006,1202,2016,2074,2508,2588,3421,3620,3760,3821,3849,5074,5087,5206,5337,6091,6212\}$;\\[-4mm]
  \item $\{0,13,110,1006,1202,2016,2074,2508,2588,3421,3620,3760,3821,3849,5074,5087,5206,5340,6091,6212\}$;\\[-4mm]
  \item $\{0,13,110,1006,1202,2016,2074,2508,2588,3421,3620,3760,3821,3849,5074,5087,5206,5343,6091,6212\}$;\\[-4mm]
  \item $\{0,13,110,1006,1202,2016,2074,2508,2588,3421,3620,3859,3872,4975,5036,5064,5206,5337,6091,6212\}$;\\[-4mm]
  \item $\{0,13,110,1006,1202,2016,2074,2508,2588,3421,3620,3859,3872,4975,5036,5064,5206,5340,6091,6212\}$;\\[-4mm]
  \item $\{0,13,110,1006,1202,2016,2074,2508,2588,3421,3620,3859,3872,4975,5036,5064,5206,5343,6091,6212\}$;\\[-4mm]
  \item $\{0,13,110,1006,1202,2022,2102,2502,2560,3418,3539,3760,3821,3849,5074,5087,5206,5340,6094,6293\}$;\\[-4mm]
  \item,$\{0,13,110,1006,1202,2022,2102,2502,2560,3418,3539,3760,3821,3849,5074,5087,5206,5343,6094,6293\}$;\\[-4mm]
  \item $\{0,13,110,1006,1202,2022,2102,2502,2560,3418,3539,3859,3872,4975,5036,5064,5206,5337,6094,6293\}$;\\[-4mm]
  \item $\{0,13,110,1006,1202,2022,2102,2502,2560,3418,3539,3859,3872,4975,5036,5064,5206,5340,6094,6293\}$;\\[-4mm]
  \item $\{0,13,110,1006,1202,2022,2102,2502,2560,3418,3539,3859,3872,4975,5036,5064,5206,5343,6094,6293\}$;\\[-4mm]
  \item $\{0,13,110,1006,1202,2022,2102,2502,2560,3421,3620,3760,3821,3849,5074,5087,5206,5337,6091,6212\}$;\\[-4mm]
  \item $\{0,13,110,1006,1202,2022,2102,2502,2560,3421,3620,3760,3821,3849,5074,5087,5206,5340,6091,6212\}$;\\[-4mm]
  \item $\{0,13,110,1006,1202,2022,2102,2502,2560,3421,3620,3859,3872,4975,5036,5064,5206,5337,6091,6212\}$;\\[-4mm]
  \item $\{0,13,110,1006,1202,2022,2102,2502,2560,3421,3620,3859,3872,4975,5036,5064,5206,5340,6091,6212\}$;\\[-4mm]
  \item $\{0,13,110,1006,1202,2022,2102,2533,2670,3421,3620,3859,3872,4975,5036,5064,5233,5256,6091,6212\}$;\\[-4mm]
  \item,$\{0,13,110,1006,1202,2022,2102,2533,2670,3421,3620,3859,3872,4975,5036,5064,5233,5337,6091,6212\}$;\\[-4mm]
  \item $\{0,13,110,1006,1202,2047,2178,2502,2560,3418,3539,3760,3821,3849,5074,5087,5181,5261,6094,6293\}$;\\[-4mm]
  \item $\{0,13,110,1006,1202,2047,2181,2502,2560,3418,3539,3760,3821,3849,5074,5087,5181,5261,6094,6293\}$;\\[-4mm]
  \item $\{0,13,110,1006,1202,2047,2181,2502,2560,3418,3539,3859,3872,4975,5036,5064,5181,5261,6094,6293\}$;\\[-4mm]
  \item $\{0,13,110,1006,1202,2047,2181,2502,2560,3421,3620,3760,3821,3849,5074,5087,5181,5261,6091,6212\}$;\\[-4mm]
  \item $\{0,13,110,1006,1202,2047,2184,2502,2560,3418,3539,3760,3821,3849,5074,5087,5181,5261,6094,6293\}$;\\[-4mm]
  \item $\{0,13,110,1015,1202,2004,2102,2499,2610,3409,3539,3760,3821,3849,4993,5060,5260,5310,6097,6239\}$;\\[-4mm]
  \item $\{0,13,110,1015,1202,2004,2102,2499,2610,3409,3539,3778,3845,4975,5036,5064,5260,5310,6097,6239\}$;\\[-4mm]
  \item $\{0,13,110,1015,1202,2004,2102,2499,2610,3424,3566,3760,3821,3849,4993,5060,5260,5310,6082,6212\}$;\\[-4mm]
  \item $\{0,13,110,1015,1202,2004,2102,2499,2610,3424,3566,3778,3845,4975,5036,5064,5260,5310,6082,6212\}$;\\[-4mm]
  \item $\{0,13,110,1015,1202,2004,2151,2499,2588,3424,3566,3760,3821,3849,4993,5060,5260,5283,6082,6212\}$;\\[-4mm]
  \item $\{0,13,110,1015,1202,2013,2102,2490,2637,3409,3539,3760,3821,3849,4993,5060,5260,5283,6097,6239\}$;\\[-4mm]
  \item $\{0,13,110,1015,1202,2101,2124,2490,2637,3409,3539,3778,3845,4975,5036,5064,5172,5261,6097,6239\}$;\\[-4mm]
  \item,$\{0,13,110,1015,1205,2019,2093,2151,2570,2670,3409,3539,3760,3849,4993,5039,5260,5283,6097,6239\}$;\\[-4mm]
  \item,$\{0,13,110,1015,1205,2019,2093,2151,2570,2670,3409,3539,3778,3824,4975,5064,5260,5283,6097,6239\}$;\\[-4mm]
  \item $\{0,13,110,1015,1205,2019,2093,2151,2570,2670,3424,3566,3760,3849,4993,5039,5260,5283,6082,6212\}$;\\[-4mm]
  \item $\{0,13,110,1015,1205,2019,2093,2151,2570,2670,3424,3566,3778,3824,4975,5064,5260,5283,6082,6212\}$;\\[-4mm]
  \item $\{0,13,110,1015,1205,2084,2184,2505,2579,2637,3409,3539,3760,3849,4993,5039,5260,5283,6097,6239\}$;\\[-4mm]
  \item $\{0,13,110,1015,1205,2084,2184,2505,2579,2637,3424,3566,3760,3849,4993,5039,5260,5283,6082,6212\}$;\\[-4mm]
  \item $\{0,13,110,1015,1205,2101,2124,2505,2579,2637,3424,3566,3760,3849,4993,5039,5243,5343,6082,6212\}$;\\[-4mm]
  \item $\{0,13,110,286,878,1113,1938,2151,2603,2806,2991,3578,3606,4191,5083,5149,5267,5932,6155,6468\}$;\\[-4mm]
  \item $\{0,13,113,285,350,937,1290,1403,2410,3224,4015,4153,4268,4326,4925,5245,6012,6231,6307,6521\}$;\\[-4mm]
  \item,$\{0,13,113,285,370,917,1290,1423,2410,3224,3995,4153,4268,4326,4925,5245,6012,6231,6307,6521\}$;\\[-4mm]
  \item $\{0,13,113,285,399,917,1261,1452,2410,3224,3995,4153,4268,4326,4925,5245,6012,6202,6307,6521\}$;\\[-4mm]
  \item $\{0,13,113,296,416,879,1295,1384,2424,3259,4011,4137,4261,4361,4960,5313,5902,6271,6309,6455\}$;\\[-4mm]
  \item $\{0,13,113,296,416,952,1295,1311,2424,3259,4084,4137,4261,4361,4960,5313,5902,6198,6309,6455\}$;\\[-4mm]
  \item $\{0,13,113,308,366,917,1276,1419,2424,3205,3995,4137,4261,4361,4906,5259,6010,6217,6309,6509\}$;\\[-4mm]
  \item $\{0,13,113,308,366,952,1241,1419,2424,3205,4030,4137,4261,4361,4906,5259,6010,6182,6309,6509\}$;\\[-4mm]
  \item $\{0,13,26,113,1014,1384,1415,1884,1990,2491,2576,2769,3043,3153,4352,4717,4935,5567,5955,6055\}$;\\[-4mm]
  \item $\{0,13,26,113,1014,1384,1415,1884,1990,2557,2653,2991,3741,4352,4983,5097,5203,5567,5893,5978\}$;\\[-4mm]
  \item $\{0,13,353,466,870,1371,1661,1911,2768,3080,3193,3583,4173,4203,5037,5395,5657,5929,6131,6175\}$;\\[-4mm]
  \item $\{0,13,353,709,878,1365,1905,2394,2560,3335,3501,3950,3963,4219,4907,5098,5389,5510,5929,5987\}$;\\[-4mm]
  \item $\{0,13,53,329,447,961,1425,2046,2449,2761,3077,3413,3636,4038,5056,5564,5892,6313,6356,6467\}$;\\[-4mm]
  \item $\{0,13,53,329,447,961,1425,2046,2684,2761,3077,3178,3636,4038,5056,5564,5892,6121,6232,6548\}$;\\[-4mm]
  \item $\{0,13,53,329,682,829,1068,2067,2613,3382,3705,3845,4118,4288,4930,5070,5276,5426,5823,6455\}$;\\[-4mm]
  \item $\{0,13,53,329,682,902,1068,2067,2613,3382,3705,3772,4045,4361,4930,5070,5203,5426,5823,6455\}$.
\end{itemize}
}

\bigskip
\bigskip
            
Currently we know $11$ non-equivalent trifferent codes with length $9$ and maximum cardinality $T(9)=27$:                   
{\tiny
\setitemize{leftmargin=*}
\begin{itemize}
  \item $\!\!\!\{0,\!121,\!242,\!3164,\!3282,\!3394,\!6325,\!6437,\!6555,\!7821,\!7915,\!7955,\!10256,\!10347,\!10378,\!11230,\!11315,\!11352,\!14886,\!14926,\!15020,\!15863,\!15900,\!15985,\!18295,\!18326,\!18417\}$;\\[-4mm]
  \item $\!\!\!\{0,\!121,\!242,\!3164,\!3282,\!3394,\!6325,\!6437,\!6555,\!7821,\!7915,\!7955,\!10256,\!10347,\!10378,\!11230,\!11315,\!11352,\!14886,\!14926,\!15020,\!15865,\!15896,\!15987,\!18293,\!18330,\!18415\}$;\\[-4mm]
  \item $\!\!\!\{0,\!121,\!242,\!3164,\!3282,\!3394,\!6325,\!6437,\!6555,\!7821,\!7915,\!7955,\!10256,\!10347,\!10378,\!11230,\!11315,\!11352,\!14886,\!14927,\!15019,\!15862,\!15900,\!15986,\!18296,\!18325,\!18417\}$;\\[-4mm]
  \item $\!\!\!\{0,\!121,\!242,\!3164,\!3282,\!3394,\!6325,\!6437,\!6555,\!7821,\!7915,\!7955,\!10256,\!10347,\!10378,\!11230,\!11315,\!11352,\!14886,\!14927,\!15019,\!15866,\!15895,\!15987,\!18292,\!18330,\!18416\}$;\\[-4mm]
  \item $\!\!\!\{0,\!121,\!242,\!3164,\!3282,\!3394,\!6325,\!6437,\!6555,\!7821,\!7915,\!7955,\!10256,\!10347,\!10378,\!11230,\!11315,\!11352,\!14886,\!14930,\!15016,\!15865,\!15897,\!15986,\!18293,\!18325,\!18420\}$;\\[-4mm]
  \item $\!\!\!\{0,\!121,\!242,\!3164,\!3282,\!3394,\!6325,\!6437,\!6555,\!7821,\!7915,\!7955,\!10256,\!10347,\!10378,\!11230,\!11315,\!11352,\!14887,\!14927,\!15018,\!15861,\!15901,\!15986,\!18296,\!18324,\!18418\}$;\\[-4mm]
  \item $\!\!\!\{0,\!121,\!242,\!3164,\!3282,\!3394,\!6325,\!6437,\!6555,\!7821,\!7915,\!7955,\!10256,\!10347,\!10378,\!11230,\!11315,\!11352,\!14887,\!14927,\!15018,\!15866,\!15894,\!15988,\!18291,\!18331,\!18416\}$;\\[-4mm]
  \item $\!\!\!\{0,\!121,\!242,\!3164,\!3282,\!3394,\!6325,\!6437,\!6555,\!7821,\!7915,\!7955,\!10256,\!10347,\!10378,\!11230,\!11315,\!11352,\!14887,\!14930,\!15015,\!15864,\!15898,\!15986,\!18293,\!18324,\!18421\}$;\\[-4mm]
  \item $\!\!\!\{0,\!121,\!242,\!3164,\!3282,\!3394,\!6325,\!6437,\!6555,\!7821,\!7915,\!7955,\!10256,\!10347,\!10378,\!11230,\!11315,\!11352,\!14890,\!14930,\!15012,\!15864,\!15895,\!15989,\!18290,\!18327,\!18421\}$;\\[-4mm]
  \item $\!\!\!\{0,\!121,\!242,\!3164,\!3282,\!3394,\!6325,\!6437,\!6555,\!7821,\!7915,\!7955,\!10256,\!10347,\!10378,\!11241,\!11281,\!11375,\!14875,\!14960,\!14997,\!15863,\!15900,\!15985,\!18295,\!18326,\!18417\}$;\\[-4mm]
  \item $\!\!\!\{0,\!121,\!242,\!3164,\!3282,\!3394,\!6363,\!6457,\!6497,\!7783,\!7895,\!8013,\!10276,\!10307,\!10398,\!11241,\!11281,\!11375,\!14875,\!14960,\!14997,\!15863,\!15900,\!15985,\!18275,\!18366,\!18397\}$                   
\end{itemize}
}

%% \bibliographystyle{alphaurl}
%% \bibliography{trifferent_codes}

\end{document}